\documentclass[12pt,a4paper]{amsart}
\usepackage{mathrsfs}

\newtheorem{theo+}              {Theorem}           [section]
\newtheorem{prop+}  [theo+]     {Proposition}
\newtheorem{coro+}  [theo+]     {Corollary}
\newtheorem{lemm+}  [theo+]     {Lemma}
\newtheorem{exam+}  [theo+]     {Example}
\newtheorem{rema+}  [theo+]     {Remark}
\newtheorem{defi+}  [theo+]     {Definition}

\newenvironment{theorem}{\begin{theo+}}{\end{theo+}}

\newenvironment{corollary}{\begin{coro+}}{\end{coro+}}
\newenvironment{lemma}{\begin{lemm+}}{\end{lemm+}}

\usepackage{amsthm}
\theoremstyle{plain} \theoremstyle{remark}

\newtheorem{example}{Example}

\def \r{\mbox{${\mathbb R}$}}
\def\E{/\kern-1.0em \equiv }

\title{ Biharmonic hypersurfaces in a conformally flat space}
\author{ Liang Tang And Ye-Lin Ou $^{*}$}
\address{ Department of
Mathematics $\&$ Physics,\newline\indent Hunan University of Technology,
\newline\indent Zhuzhou Hunan,\newline\indent P. R. China\newline
\vskip0.1cm\indent Department of Mathematics \newline\indent Guangxi University for Nationalities
\newline\indent Nanning 530006,\newline\indent P. R. China}
\thanks{$^{*}$Research supported by
NSF of Guangxi (P. R. China), 2011GXNSFA018127.}
\begin{document}

\date {04/20/2012} 
\maketitle
\section*{Abstract}
\begin{quote}
{\footnotesize } Biharmonic hypersurfaces in a generic conformally
flat space are studied in this paper. The equation of such
hypersurfaces is derived and is used to determine the conformally
flat metric  $f^{-2}\delta_{ij}$ on the Euclidean space $\r^{m+1}$
so that  a minimal hypersurface $M^m\longrightarrow (\r^{m+1},
\delta_{ij})$ in a Euclidean space becomes a biharmonic hypersurface
$M^m\longrightarrow (\r^{m+1}, f^{-2}\delta_{ij})$ in the
conformally flat space. Our examples include all biharmonic
hypersurfaces found in \cite{Ou1} and \cite{OT} as special cases.
\end{quote}
\section{Biharmonic maps and submanifolds}
All manifolds, maps, and tensor fields studied in this paper are assumed to be smooth.\\

A  map $\varphi:(M, g)\longrightarrow (N, h)$
between Riemannian manifolds is called a {\em  biharmonic map} if it is a critical point of the
bienergy functional
\begin{equation}\nonumber
E^{2}\left(\varphi,\Omega \right)= \frac{1}{2} {\int}_{\Omega}
\left|\tau(\varphi) \right|^{2}{\rm d}x
\end{equation}
for every compact subset $\Omega$ of $M$, where $\tau(\varphi)={\rm
Trace}_{g}\nabla {\rm d} \varphi$ is the tension field of $\varphi$.
The {\em biharmonic map equation} is the Euler-Lagrange equation of
this functional which can be written as (see \cite{Ji1})
\begin{equation}\label{BTF}
\tau^{2}(\varphi):={\rm
Trace}_{g}(\nabla^{\varphi}\nabla^{\varphi}-\nabla^{\varphi}_{\nabla^{M}})\tau(\varphi)
- {\rm Trace}_{g} R^{N}({\rm d}\varphi, \tau(\varphi)){\rm d}\varphi
=0,
\end{equation}
where $R^{N}$ denotes the curvature operator of $(N, h)$ defined by
$$R^{N}(X,Y)Z= [\nabla^{N}_{X},\nabla^{N}_{Y}]Z-\nabla^{N}_{[X,Y]}Z.$$

A submanifold $M$ of $(N,h)$ is called a {\em biharmonic submanifold}
if its inclusion map ${\bf i}: (M, {\bf i}^*h)\longrightarrow (N,h)$
is a biharmonic isometric immersion. \\

From the well-known fact that a harmonic map is a map between Riemannian manifold whose tension field $\tau(\varphi)={\rm
Trace}_{g}\nabla {\rm d} \varphi$ vanishes identically and that an
isometric immersion is minimal if and only if it is harmonic we have the following relationships:
$$\{{\rm \bf Harmonic\; maps}\}\subset \{{\rm \bf Biharmonic \;maps}\},$$
$$\{{\rm \bf Minmal\; submanifolds}\}\subset \{{\rm \bf Biharmonic \;submanifolds}\}.$$
These relationships justify our using the names {\em proper biharmonic maps}  for those biharmonic maps which are not harmonic and  {\em proper biharmonic submanifolds} for those biharmonic submanifolds which are not minimal.\\

\indent Among the interesting problems in the study of biharmonic submanifolds are the following two conjectures.\\

{\em Chen's conjecture} (see e.g., \cite{Ch1}, \cite{CI}, \cite{HV}, \cite{Di}, \cite{CMO1}, \cite{Ch2}, \cite{NU} and
the references therein): any biharmonic submanifold in
a Euclidean space is minimal.\\

{\em The generalized Chen's conjecture}: any biharmonic submanifold
of $(N, h)$ with ${\rm Riem}^N\leq 0$ is minimal (see e.g.,
 \cite{CMO1},  \cite{BMO1}, \cite{BMO2}, \cite{IIU}, \cite{NU}, \cite{Ch2}, \cite{OT} and the references therein).\\

While Chen's conjecture is still open the generalized Chen's
conjecture has been proved to be false in the authors recent paper
\cite{OT}. The main idea in solving the generalized Chen's
conjecture is to construct a proper bihmarmonic hypersurface in a
$5$-dimensional conformally flat space with negative sectional
curvature. This and the biharmonicity of the product maps (see
\cite{Ou2}) are used to construct many examples of proper biharmonic
submanifolds in a nonpositively curved manifolds.\\

\indent A Riemannian manifold ($M^{m}, g$) is called a conformally
flat space if for any point of $M$ there exists a neighborhood which
is conformally diffeomorphic to an open subset of the Euclidean
space $\r^{m}$. More precisely, $\forall\,p\in M$, there exists a
neighborhood $U$, $p\in\,U\subset\,M$ and a diffeomorphism
$\varphi:(U, g)\longrightarrow \varphi(U)\subset (\r^m, h)$, such
that\, ${\varphi}^{*} h= e^{2\sigma}g$, where $h$ denotes the
standard Euclidean metric on $\r^m$. The following facts are well
known:
\begin{itemize}
\item any two-dimensional Riemannian manifold is conformally flat;
\item $\{{\rm Space\; forms}\}\subset \{{\rm Conformally\; flat \;spaces}\}$;
\item $\mathbb{S}^m\times \r$ and $\mathbb{H}^m\times \r$ are conformally flat spaces.
\end{itemize}

In the context of biharmonic submanifolds, there has been a growing
study on biharmonic submanifolds in space forms (see \cite{Ch1},
\cite{Ch2}, \cite{CI}, \cite{Ji1}, \cite{Ji2}, \cite{CMO1},
\cite{CMO2}, \cite{BMO1}, \cite{CMO2}, \cite{Di}, \cite{HV}, and the
references therein) in recent years. Some interesting examples of
biharmonic hypersurfaces in some special conformally flat spaces and
their applications in solving the generalized Chen' s conjecture
have been obtained in \cite {Ou1} and \cite{OT}, and some
classifications of biharmonic submanifolds in conformally flat
spaces $\mathbb{S}^m\times \r$ and $\mathbb{H}^m\times \r$ have been
given in \cite{OW}, \cite{FOR}, and \cite{FR}.\\

This paper attempts to study biharmonic hypersurfaces in a generic
conformally flat space. We derived the equation for such
hypersurfaces which generalizes the equation for the biharmonic
hypersurface in a space form. As an application, we use the equation
to determine the conformally flat metric  $f^{-2}\delta_{ij}$ on the
Euclidean space $\r^{m+1}$ so that  a minimal hypersurface
$M^m\longrightarrow (\r^{m+1}, \delta_{ij})$ in a Euclidean space
becomes a biharmonic hypersurface $M^m\longrightarrow (\r^{m+1},
f^{-2}\delta_{ij})$ in the conformally flat space. Our examples
include all biharmonic hypersurfaces found previously in \cite{Ou1}
and \cite{OT} as special cases.\\

\section{Biharmonic hypersurfaces in a conformally flat space}

Biharmonic hypersurfaces in a generic Riemannian manifold has been
studied in \cite{Ou1} and one of the main results is the following
\begin{theorem}\label{EQ1}\cite{Ou1}
Let $\varphi:M^{m}\longrightarrow N^{m+1}$ be an isometric immersion
of codimension-one with mean curvature vector $\eta=H\xi$. Then
$\varphi$ is biharmonic if and only if:
\begin{equation}\label{eq1}
\begin{cases}
\Delta H-H |A|^{2}+H{\rm
Ric}^N(\xi,\xi)=0,\\
 2A\,({\rm grad}\,H) +\frac{m}{2} {\rm grad}\, H^2
-2\, H \,({\rm Ric}^N\,(\xi))^{\top}=0,
\end{cases}
\end{equation}
where ${\rm Ric}^N : T_qN\longrightarrow T_qN$ denotes the Ricci
operator of the ambient space defined by $\langle {\rm Ric}^N\, (Z),
W\rangle={\rm Ric}^N (Z, W)$ and  $A$ is the shape operator of the
hypersurface with respect to the unit normal vector $\xi$.
\end{theorem}

When the ambient space is a conformally flat space we have

\begin{theorem}\label{EQU}
Let $\varphi:(M^m,g)\longrightarrow (C^{m+1}, h=e^{-2\sigma}\bar h)$ be an isometric
immersion of codimension-one with mean curvature vector $\eta=H\xi$,
where $(C^{m+1}, h)$ denotes a conformally flat space.Then $\varphi$
is biharmonic if and only if:
\begin{equation}\label{equ}
\begin{cases}
\Delta^M H-H |A|^{2}+H\{\Delta_h\sigma+(m-1)[\rm Hess_h(\sigma)(\xi,
\xi)-(\xi\sigma)^2+|{\rm grad}_h\sigma|_h^{2}] \}=0,\\
 2A\,({\rm grad_g}\,H) +\frac{m}{2} {\rm grad_g}\, H^2
-2(m-1)\, H \,[{\rm grad}_g(\xi\sigma)-(\xi\sigma){\rm grad}_g
\sigma+ A({\rm grad}_g \sigma)]=0.
\end{cases}
\end{equation}
where $\sigma$\;is the conformal factor of $(C^{m+1}, h)$,\;$A$ is
the shape operator of the hypersurface with respect to the unit
normal vector $\xi$,\;${\rm grad}_h,\; \Delta_h$ and ${\rm grad}_g,\; \Delta^M$ denote the
gradient and the Laplacian of ambient space $C^{m+1}$ and the hypersurface $M$ respectively.
\end{theorem}
\begin{proof}  Let $\nabla, \rm R,  Ric,  grad_h, \Delta_h$ denote the Levi-Civita connection, Riemannian curvature, Ricci curvature, the gradient operator, and the Laplace operator on $(C^{m+1}, h)$ respectively. The same notations with a ``$\bar{\;\;}$"  over head will be used for the counterparts on  $(C^{m+1}, \bar h)$ . It is well known (see e.g., \cite{Wa}) that the relationship between of the
two connections  and the Riemannian curvature are given by
\begin{equation}\label{a}
{\bar \nabla}_{X}Y=\nabla_{X}Y+(X\sigma)Y+(Y\sigma)X-h(X,Y){\rm
grad}_h\sigma,
\end{equation}
\begin{eqnarray}
{\bar {\rm R}}(W,Z, X,Y)&=&e^{2\sigma}\{{\rm R}(W, Z, X, Y)+h(\nabla_{X}{\rm
grad}_h\sigma,Z)h(Y, W)\\\notag&&-h(\nabla_{Y}{\rm
grad}_h\sigma,Z)h(X, W)
+h(X,Z)h(\nabla_{Y}{\rm
grad}_h\sigma, W)\\\notag &&-h(Y,
Z)h(\nabla_{X}{\rm
grad}_h\sigma ,W)+[(Y\sigma)(Z\sigma)\\\notag
&&-h(Y,Z)|{\rm
grad}_h\sigma|^{2}]h(X,W)
-[(X\sigma)(Z\sigma)\\\notag&&-h(X,Z)|{\rm
grad}_h\sigma|^{2}]h(Y,W)
+[(X\sigma)h(Y,Z)\\\notag&&-(Y\sigma)h(X,Z)]h({\rm
grad}_h\sigma,W)\}.
\end{eqnarray}

In local coordinates, we have
\begin{eqnarray}\notag
e^{-2\sigma}{\bar
{\rm R}}_{ij\,kl}&=&{\rm R}_{ij\,kl}+h_{il}\sigma_{jk}-h_{ik}\sigma_{jl}+h_{jk}
\sigma_{il}-h_{jl}\sigma_{ik} +(h_{il}h_{jk}-h_{ik}h_{jl})|{\rm
grad}_h\sigma|^{2},
\end{eqnarray}
where we have used the  notation
$\sigma_{jk}=\nabla_{k}\sigma_j-\sigma_k\sigma_j=\nabla_{k}\nabla_j\sigma-\sigma_k\sigma_j={\rm
Hess}_h(\sigma)(\partial_j, \partial_k)-\sigma_k\sigma_j$ with
${\rm Hess}_h$ denoting the Hessian operator.\\

The relationship between the Ricci curvatures of the two conformally related metrics is given by
\begin{equation}\label{b}
{\overline {\rm Ric}}_{jk}={\rm Ric}_{jk}-(m-1)\sigma_{jk}-h_{jk}[ \Delta_h\sigma
+(m-1)|{\rm
grad}_h\sigma|^{2}].
\end{equation}

It follows that if the metric $\bar h$ is Euclidean (flat), i.e., the metric $h=e^{-2\sigma}\bar h$ is conformally flat, then we obtain the Ricci curvature of a conformally flat space
\begin{eqnarray}\label{1.1}
{\rm Ric}_{jk}&=&(m-1)\sigma_{jk}+h_{jk}[ \Delta\sigma +(m-1)|{\rm
grad}_h\sigma|^{2}]\\\notag &=&(m-1){\rm Hess}_h(\sigma)(\partial_j,
\partial_k)-(m-1)\sigma_k\sigma_j\\\notag&&+h_{jk}[
\Delta_h\sigma +(m-1)|{\rm
grad}_h\sigma|^{2}].
\end{eqnarray}

For a conformally flat space $(C^{m+1}, h)$, we can choose local coordinates \\$(x^{1},  x^2,\cdots, x^{m+1})$ with local frame $\{\partial_{\alpha}=\frac{\partial}{\partial x^{\alpha}}\}_{\alpha=1, \cdots, m+1}$. It follows that $\{\epsilon_{\alpha}=e^{\sigma}\partial_{\alpha}\}_{\alpha=1, \cdots, m+1}$ form an
orthonormal basis on $(C^{m+1}, h=e^{-2\sigma}\bar{h})$.  Let $\{e_{1}, \cdots,
e_{m}, \xi\}$ be an orthonormal frame on $(C^{m+1},  h)$ adapted to the hypersurface $M$ so that $\{e_{i}\}_{i=1, \cdots, m}$ tangent  and $\xi$ normal to $M$.
The relationship between the two orthonormal frames is given by
\begin{eqnarray}\label{matrix}
\begin{cases}
e_i=T_i^{\alpha}\epsilon_{\alpha},\; i=1, 2, \cdots, m.\\
\xi=T_{m+1}^{\alpha}\epsilon_{\alpha},
\end{cases}
\end{eqnarray}
 where $(T_{\alpha}^{\beta})$ is an $(m+1)\times (m+1)$ orthogonal matrix.\\

Using the relation (\ref{matrix}) we can compute the Ricci curvatures
\begin{eqnarray}\label{c}
{\rm Ric}(\xi, \xi)&=& e^{2\sigma}T_{m+1}^{j}T_{m+1}^{k}{\rm Ric}_{jk}\\\notag
&=&\Delta_h\sigma+(m-1)[{\rm Hess}_h(\sigma)(\xi,
\xi)-(\xi\sigma)^2+|{\rm grad}_h\sigma|^{2}],
\end{eqnarray}
 and
\begin{eqnarray}\label{d}
[{\rm Ric}(\xi)]^T&=& \sum_{i=1}^{m}e^{2\sigma}T_{m+1}^{j}T_{i}^{k}{
\rm Ric}_{jk}e_i\\\notag &=&(m-1)[{\rm
grad}_g(\xi\sigma)+\sum_{i=1}^{m} (Ae_i)(\sigma)e_i-\xi(\sigma){\rm
grad}_g \sigma]
\\\notag &=&(m-1)[{\rm grad}_g(\xi\sigma)-(\xi\sigma){\rm
grad}_g\sigma+\sum_{i,j=1}^{m}b(e_i,e_j)e_j(\sigma)e_i]
\\\notag &=&(m-1)[{\rm grad}_g(\xi\sigma)-(\xi\sigma){\rm grad}_g\sigma+A({\rm grad}_g \sigma)].
\end{eqnarray}
\indent Substituting (\ref{c}) and (\ref{d}) into the biharmonic
hypersurface equation (\ref{eq1}) we obtain the theorem.
\end{proof}

\indent As immediate consequences of Theorem \ref{EQU}  we have \\
\begin{corollary}\label{C2}
A constant mean curvature hypersurface in a conformally flat space $C^{m+1}$ is proper biharmonic if and only
if
\begin{equation}\label{C3}
\begin{cases}
 |A|^{2}=\Delta_h\sigma+(m-1)[\rm Hess_h(\sigma)(\xi,
\xi)-(\xi\sigma)^2+|\rm grad_h\sigma|_h^{2}] ,\\
\rm grad_g\xi(\sigma)-\xi(\sigma)\rm grad_g \sigma+ A(\rm
grad_g \sigma)=0.
\end{cases}
\end{equation}
In particular, if $\xi(\sigma)=0,(\ref{C3})$ reduces to
\begin{equation}\label{C4}
\begin{cases}
 |A|^{2}=\Delta_h\sigma+(m-1)[\rm Hess_h(\sigma)(\xi,
\xi)+|\rm grad_h\sigma|_h^{2}] ,\\
A(\rm grad_g \sigma)=0.
\end{cases}
\end{equation}
\end{corollary}
\indent Let $M$ be a totally umbilical hypersurface in $C^{m+1}$,
i.e., all principal normal curvature at any point $p\in M$ are equal
to the same number $\lambda(p)$, it follows that
\begin{eqnarray*}
H=\frac{1}{m}\sum_{i=1}^{m}h(Ae_i,e_i)=\lambda,\\
A({\rm grad_g} H)=A(\sum_{i=1}^{m}(e_i\lambda)e_i)=\frac{1}{2}{\rm
grad_g}
H^2,\\
|A|^2=m\lambda^2=mH^2,\;A({\rm grad_g} \sigma)=\lambda {\rm grad_g}
\sigma=H{\rm grad_g} \sigma.
\end{eqnarray*}
From this we have
\begin{corollary}\label{C5}A totally umbilical hypersurface in
$C^{m+1}$ is biharmonic if and only if its mean curvature function
$H$ is a solution of the following PDEs
\begin{equation}\label{C6}
\begin{cases}
\Delta^M H-mH^{3}+H\{\Delta\sigma+(m-1)[\rm Hess(\sigma)(\xi,
\xi)-(\xi\sigma)^2+|\rm grad\sigma|^{2}] \}=0,\\
 \frac{2+m}{2} {\rm grad_g}\, H^2
-2(m-1)\, H \,[\rm grad_g(\xi\sigma)-(\xi\sigma)\rm grad_g \sigma+ H
\rm grad_g \sigma]=0.
\end{cases}
\end{equation}
if $H=\xi(\sigma),(\ref{C6})$ becomes
\begin{equation}\label{C7}
\begin{cases}
\Delta^M H-mH^{3}+H\{\Delta\sigma+(m-1)[\rm Hess(\sigma)(\xi,
\xi)-H^2+|\rm grad\sigma|^{2}] \}=0,\\
(m-4){\rm grad_g}\, H^2 =0.
\end{cases}
\end{equation}
\end{corollary}

In \cite{Ou1}, an interesting foliation of proper biharmonic
hypersurfaces was found in a conformally flat space, and later in
\cite{OT}, counter examples to the generalized Chen's conjecture
were found by using proper biharmonic hypersurfaces in
$5$-dimensional conformally flat spaces. All those proper biharmonic
hypersurfaces were constructed by starting with a hyperplane (a
totally geodesic hypersurface) in a Euclidean space and then
performing a suitable conformal change of the Euclidean metric into
a conformally flat metric which renders the totally geodesic
hypersurface a proper biharmonic hypersurface. Now we study the more
general problem: given a minimal hypersurface in a Euclidean space,
i.e., a minimal isometric immersion, $\phi: M^m\hookrightarrow
(\r^{m+1}, h)$, under what conditions on  $ f$ the hypersurface in
the conformally flat space $\phi: M^m\hookrightarrow (\r^{m+1},
f^{-2}h)$ becomes a proper biharmonic hypersurface?\\

\begin{theorem}\label{MT2}
Let $\phi: M^m\hookrightarrow (\r^{m+1}, h)$ be a minimal
hypersurface with the unit normal vector field $\xi$ in a Euclidean
space, then, the hypersurface $\phi: M^m\hookrightarrow (\r^{m+1},
f^{-2}h)$ in the conformally flat space is a biharmonic hypersurface
if and only if
\begin{equation}\label{ME}
\begin{cases}
f\Delta_g(f(\xi f))-m {\rm grad}_gf(f(\xi f))-f^2(\xi f)|A|_h^2-2m(\xi f)^3\\\;\;\;\;\;\;\;\;\;\;\;\;\;+mf(\xi f) {\rm Hess}_h(f)(\xi, \xi)=0,\\
 2f A({\rm
grad}_g(\xi f))-2(m-1)(\xi f)A({\rm grad}_gf)+(4-m)(\xi f){\rm
grad}_g(\xi f)=0,
\end{cases}
\end{equation}
where $h$ denotes the standard Euclidean metric on $\r^{m+1}$ and
$g$ and $A$ are the metric and the shape operator of the minimal
hypersurface $M^m\longrightarrow (\r^{m+1}, h)$.
\end{theorem}
\begin{proof}

First of all, with the new notations, the equation (\ref{equ}) for biharmonic hypersurface in the  conformally flat space $\phi: M^m\hookrightarrow (\r^{m+1}, \bar h=f^{-2}h)$ becomes
\begin{equation}\label{GD06}
\begin{cases}
\Delta_{\bar g}^M \bar {H}-\bar {H} |A_{\bar h}|_{\bar h}^{2}+\bar {H}\{\Delta_{\bar h}\sigma+(m-1)[\rm Hess_{\bar h}(\sigma)({\bar \xi},
{\bar \xi})-({\bar \xi} \sigma)^2+|{\rm grad}_{\bar h}\sigma|_{\bar h}^{2}] \}=0,\\
 2A\,({\rm grad}_{\bar g}\,\bar {H}) +\frac{m}{2} {\rm grad}_{\bar g}\, \bar {H}^2
-2(m-1)\, \bar {H} \,[{\rm grad}_{\bar g}{\bar \xi}(\sigma)-{\bar \xi}(\sigma){\rm grad}_{\bar g}
\sigma+ A({\rm grad}_{\bar g} \sigma)]=0,
\end{cases}
\end{equation}
where $\sigma=\ln f$.

{\bf Claim:} Let $\phi: M^m\hookrightarrow (N^{m+1}, h)$ be a minimal hypersurface  with the unit normal vector field $\xi$. Then, the
mean curvature function ${\bar H}$ of $M $ as a hypersurface
 in $\phi: M^m\hookrightarrow (N^{m+1}, {\bar h}=f^{-2}h)$ is given by
\begin{equation}\label{GD07}
{\bar H}=\xi f.
\end{equation}
{\bf Proof of the Claim:} Let $\phi: M^m\hookrightarrow (N^{m+1}, h)$ be a hypersurface  with mean curvature function $H$ and the unit normal vector field $\xi$.  Let ${\bar h}=f^{-2}h$ be a conformal change of the metric. Then, a straightforward computation (see e.g., \cite{Ne}) the
mean curvature function $\bar {H}$ of $M $ as a hypersurface
 in $(N^{m+1}, {\bar h})$ is given by
\begin{equation}\label{R1}
{\bar H}=f H+f\xi(\ln f),
\end{equation}
from which we obtain the  Claim.\\

Let $\{e_1, \cdots, e_m, \xi\}$ be an orthonormal frame adapted
to the minimal hypersurface $\phi: M^m\hookrightarrow (\r^{m+1}, h)$ with $\xi$ being the unit normal vector with respect to the metric $h$. Then,
$\{\bar e_i=fe_i, \; i=1, \cdots, m; \bar \xi=f\xi\}$ consist an orthonormal frame adapted
to the hypersurface $\phi: M^m\hookrightarrow (\r^{m+1}, \bar h=f^{-2}h)$ with respect to the metric $\bar h$.\\

\indent Using  the relationship $(\ref{a})$ between the
connections of the two conformally related metrics we have
\begin{eqnarray}
{\bar \nabla}_{\bar e_i} \bar e_j&=&f^2\nabla_{e_i} e_j-e_j(f)\bar
e_i+\delta_{ij}{\rm grad}_{\bar h} \ln f,\\\label{Hre2}
{\bar \nabla}_{\bar e_i} \bar \xi&=&f^2\nabla_{e_i} \xi-\xi(f)\bar
e_i.
\end{eqnarray}
A straightforward computation yields
\begin{eqnarray}\label{GD08}
|A|_{\bar h}^2&=&\Sigma_{i=1}^m\bar h({\bar \nabla}_{\bar e_i}\bar
\xi,{\bar \nabla}_{\bar e_i}\bar \xi)
\\\notag&=&\Sigma_{i=1}^m [f^2h(\nabla_{e_i}
\xi,\nabla_{e_i} \xi)]+2mf\xi(f)H+m(\xi f)^2\\\notag
&=&f^2|A|_h^2+m(\xi f)^2.
\end{eqnarray}

\indent Noting that $\sigma=\ln f$ we have:
\begin{eqnarray}\notag
\bar e_i(\sigma)&=&fe_i(\ln f)=e_i(f), \\\notag
\bar e_i\bar
e_i(\sigma)&=&fe_i(e_i(f))=fe_ie_i(f), \\\notag
\bar \xi(\sigma)&=&\xi(f), \\\notag {\rm grad}_{\bar
h}\sigma&=&\sum_{i=1}^{m+1}\bar e_i(\ln f)\bar e_i =f{\rm grad}_hf, \\\label{GD09}
|{\rm grad}_{\bar h}\sigma|^2_{\bar h}&=&|{\rm grad_h}f|_h^2.\\\notag
 {\bar\nabla}_{\bar e_i}\bar e_i(\sigma)&=&f\nabla_{e_i}
e_i(f)-e^2_i(f)+|{\rm grad}_hf|^2, \\\notag
\bar \xi \bar \xi(\sigma)&=&f \xi \xi(f),\\\notag{\bar \nabla}_{\bar
\xi}\bar
\xi(\sigma)&=&f \nabla_{\xi} \xi(f)-(\xi f)^2+|{\rm grad}_hf|^2,\\\notag
\bar e_i\bar e_i(\bar H)&=&fe_i(fe_i(\bar H))=fe_i(f)e_i(\bar
H)+f^2e_ie_i(\bar H), \\\notag
{\bar \nabla}_{\bar e_i}\bar e_i(\bar H)&=&f^2\nabla_{e_i} e_i(\bar
H)-fe_i(f)e_i(\bar H)+f{\rm grad}_h\bar H(f).
\end{eqnarray}

A further computation gives
\begin{eqnarray}\notag
 {\bar\nabla}^M_{\bar e_i}\bar e_i(\bar H)&=&{\bar \nabla}_{\bar e_i}\bar
e_i(\bar H)-b_{\bar h}(\bar e_i, \bar e_i)\bar \xi(\bar H),\\\label{GD10}
\Delta^M_{\bar h}{\bar H}&=&f\Delta_g(f\xi f)-f(\xi f)\Delta_g{f}-m f({\rm
grad}_g f)(\xi f),\\\label{GD11}
\Delta_{\bar  h}\sigma&=&f\Delta_h f-m|{\rm
grad}_hf|^2,\\\label{GD12}
{\rm Hess}_{\bar h}(\sigma)(\bar \xi,\bar \xi)&=&\bar \xi\bar
\xi(\sigma)-{\bar \nabla}_{\bar \xi}\bar \xi(\sigma)\\\notag
&=&f{\rm Hess}_h(f)(\xi, \xi)+({\bar \xi} \,\sigma)^2-|{\rm grad}_{\bar
h}\sigma|^2,\\\notag
{\rm grad}^M_{\bar h}\bar H&=&\sum_{i=1}^{m}\bar e_i(\xi f)\bar
e_i=f^2{\rm grad}_g(\xi f), \\\label{GD13}
A_{\bar h}({\rm grad}^M_{\bar h}{\bar H})&=&f^3A({\rm grad}_g (\xi f) )+f^2\xi(f){\rm
grad}_g(\xi f),\\\label{GD14}
{\rm grad}^M_{\bar h}\sigma&=&\sum_{i=1}^{m}\bar e_i(\ln f)\bar
e_i=f\sum_{i=1}^{m}e_i(f)e_i=f{\rm grad}_gf,\\\label{GD15}
A_{\bar h}({\rm grad}^M_{\bar h}\sigma)&=&f^2A({\rm grad}_gf)+f\xi(f){\rm grad}_gf.
\end{eqnarray}
 Substituting (\ref{GD07}),  (\ref{GD08}),  (\ref{GD09}),  (\ref{GD10}),  (\ref{GD11}),  (\ref{GD12}),  (\ref{GD13}),  (\ref{GD14}),  (\ref{GD15})  into  (\ref{GD06}) and simplifying the resulting equation we obtain Equation (\ref{ME}) which completes the proof of the theorem.
\end{proof}

As an application of Theorem \ref{MT2}  we have the following theorem which give a generalization of Theorem 3.1 in \cite{Ou1}.
\begin{theorem}\label{f31}
The hyperplane $\varphi : \mathbb{R}^m\rightarrow (\mathbb{R}^{m+1},\bar
h=f^{-2}(x_1,\ldots, x_m,z)(\sum_{i=1}^m{\rm d}{x_i}^{2}+{\rm
d}{z}^{2})$ with $\varphi(x_1,\ldots, x_m)=(x_1,\ldots, x_m,c)$ in the conformally flat space
 is biharmonic  if and only if one of the following three cases happens\\
$1)$ it is minimal i.e., $f_{z}=0$; \\
$2)$ $m=4$ and f is solution of the equation
\begin{eqnarray}\label{f32}
\sum_{i=1}^{4}[f^2f_{iiz}-2ff_if_{iz}+ff_zf_{ii}-4f_zf_i^2]+4f_z(ff_{zz}-2f_z^2)=0,
\end{eqnarray}
$3)$ The hyperplane has nonzero constant mean curvature and $f$
takes the form $f(x_1, \ldots, x_m, z)=p(x_1, \ldots, x_m)+q(z)$
with $p$ and $q$ satisfying the following equation
\begin{eqnarray}\label{35}
(p\sum_{i=1}^{m}p_{ii}-m\sum_{i=1}^{m}p_i^2)+m(qq_{zz}-2q_z^2)+(m
pq_{zz}+q\sum_{i=1}^{m}p_{ii})=0.
\end{eqnarray}
In particular, if $p=\rm constant$, then $f(x_1,\ldots,
x_m,z)=(Az+B)^{-1}$.
\end{theorem}
\begin{proof}
Noting that the hyperplane in the Euclidean space $\varphi : \mathbb{R}^m\rightarrow (\mathbb{R}^{m+1},
h=\sum_{i=1}^m{\rm d}{x_i}^{2}+{\rm
d}{z}^{2})$  with $\varphi(x_1,\ldots, x_m)=(x_1,\ldots, x_m, c)$ is a totally geodesic hypersurface with
unit normal vector field $\xi=\partial_z$ we have  $\xi f=f_z$, $A=0$, and that the induced metric on the
hyperplane is the standard Euclidean metric $g=\sum_{i=1}^m{\rm d}{x_i}^{2}$. It follows that the second
equation of (\ref{ME}) is equivalent to
\begin{equation}
(m-4)f_zf_{zi}=0,\;\; i=1, 2, \cdots, m.
\end{equation}
Case I: $f_z=0.$ This means the conformally flat metric is actually homothetic to the Euclidean metric
and the hyperplane remains to be totally geodesic.\\
Case II: $f_{zi}=0, \; i=1, 2, \cdots, m$. In this case, the mean curvature of the hypersurface
is given by $H=(\xi f)|_{z=C}=\rm constant$. It follows that  $f$ takes the form $p(x_1, \cdots, x_m)+q(z)$
and the first equation of (\ref{ME}) becomes
\begin{eqnarray}\label{35}
(p\sum_{i=1}^{m}p_{ii}-m\sum_{i=1}^{m}p_i^2)+m(qq_{zz}-2q_z^2)+(m
pq_{zz}+q\sum_{i=1}^{m}p_{ii})=0.
\end{eqnarray}

In particular, if $p=\rm constant$, then $f$ depends on $z$ alone
and by Theorem
3.1 in \cite{Ou1} we conclude that $f(z)=(Az+B)^{-1}$.\\

Case III: $m=4$. In this case, one can easily (use the fact that both $g$ and $h$ are Euclidean metrics)
that the first equation of (\ref{ME}) is equivalent to (\ref{f32}). Summarizing the above results we obtain the theorem.
\end{proof}
\begin{example}
 For any $A,B,C,D>0$, let $M^4=\{(x_1, \cdots, x_4)\in \r^4|\; x_i>0, \;i=1, 2, 3, 4\}$ and $N^5= M\times \r^+$.
 Then, the
isometric immersion $\varphi : M^4\rightarrow
(N^5, h=(Ax_i+B)(Cz+D)^2(\sum_{i=1}^m{\rm
d}{x_i}^{2}+{\rm d}{z}^{2}))$ with
$\varphi(x_1,\cdots,x_4)=(x_1, \cdots, x_4, K)$ and $K>0$ is a proper biharmonic hyper surface
in the conformally flat space.\\

In fact, this can be obtained by looking for a solution of  (\ref{f32}) of the form $f=p(x_i)q(z)$.
In this case, one can easily check that
\begin{eqnarray*}
f_i=p'(x_i)q(z),\;f_z=p(x_i)q'(z),\;f_{ii}=p''(x_i)q(z),\\f_{iz}=p'(x_i)q'(z),\;f_{zz}=p(x_i)q''(z),\;f_{iiz}=p''(x_i)q'(z).
\end{eqnarray*}
Substituting these  into the equation (\ref{f32}) we have
\begin{eqnarray*}
p(x_i)q'(z)q^2(z)[p(x_i)p''(x_i)-3(p'(x_i))^2]+2p^3(x_i)q'(z)[q(z)q''(z)-2(q'(z))^2]=0,
\end{eqnarray*}
\indent By looking for the solutions satisfying
$p(x_i)p''(x_i)-3(p'(x_i))^2=0$ and $q(z)q''(z)-2(q'(z))^2=0$ we
obtain special solutions $p(x_i)=\frac{1}{\sqrt{Ax_i+B}}$ and
$q(z)=\frac{1}{Cz+D}$ with positive constants $A,B,C,D$. From this we
obtain the example.
\end{example}
\indent Recall that the main idea used in constructing a counter
example to the generalized Chen's conjecture (see \cite{OT}) is to
find a conformal change of the Euclidean metric on $\r^{m+1}$ so
that certain hyperplane becomes a proper biharmonic hypersurface and
at the same time the conformally flat metric has nonpositive
sectional curvature. In \cite{Ou1}, the  author starts with a plane
perpendicular to the last coordinate axis and searches for a
conformally flat metric whose conformal factor depends only on the
last coordinate. It turns out that for this type of the metric the
hyperplane does become proper biharmonic but the metric cannot be
nonpositively curved. Later in \cite{OT}, we succeed in finding many
counter examples by using a hyperplane in a more general position
and searching for the same type of the conformally flat metrics. Our
next theorem shows that even we starts with a hyperplane that
perpendicular to the last coordinate axis we can find conformally
flat metrics depending on more variables which then give negative
sectional curvature and turn the hyperplane into a proper biharmonic
hypersurface.\\

Before stating and proving our next theorem we prove the following lemma which has its own interest.
\begin{lemma}\label{37}
Let $(C^{m}, h=f^{-2}(x_1,\cdots, x_m)\sum_{i=1}^{m}{\rm
d}{x_i}^{2})$ be a conformally flat space and let $\{e_i=f\frac{\partial}{\partial
x_i},\;i=1, 2,\cdots, m\}$ denote an orthonormal
frame on $C^m$.  Let $P$ be a plan section at a
point and suppose that $P$ is spanned by an orthonormal basis
$X, Y$. Then, the sectional curvature $K(P)$ of $C^{m}$ is given by
\begin{eqnarray}\label{38}
K(P)=\sum_{i, j=1}^{m}(a_ia_j+b_ib_j)ff_{ij}-\sum_{i=1}^{m}f_i^2.
\end{eqnarray}
where $a_i=h(X,\,e_i),\;b_i=h(Y,\,e_i),\;f_i=\frac{\partial
f}{\partial x_i},\;f_{ij}=\frac{\partial^2 f}{\partial x_i\partial
x_j}$.
\end{lemma}
\begin{proof}
\indent One can easily check  (see, e.g., \cite{Wa} or \cite{OT}) that the sectional curvature $K(P)$ of
$C^{m}$ is given by
\begin{eqnarray}\notag \label{SC}
K(P)&=&h(\nabla_X {\rm grad}\sigma,X)+h(\nabla_Y {\rm
grad}\sigma,Y)+|{\rm grad}\sigma|^2-(X \sigma)^2-(Y \sigma)^2\\\label{SC}
&=&
XX(\sigma)-(\nabla_XX)(\sigma) -(X \sigma)^2\\\notag &&+YY(\sigma)-(\nabla_YY)(\sigma)-(Y \sigma)^2 +|{\rm
grad}\sigma|^2,
\end{eqnarray}
where $\sigma=\ln f$.\;By hypothesis and a straightforward computation we have
\begin{eqnarray}\notag
\sum_{i=1}^{m }a_i^2&=&1,\;\;\;\;\sum_{i=1}^{m}b_i^2=1,\\\notag e_i(\sigma)&=&f_i,\\\notag
\nabla_{e_i}e_j&=&\delta_{ij}\sum_{k=1}^{m}f_ke_k-f_je_i,\\\label{gd50} |{\rm
grad}\sigma|^2&=&\sum_{i=1}^{m}f_i^2.
\end{eqnarray}
A further computation gives
\begin{eqnarray}\notag
X(\sigma)&=&\sum_{i=1}^{m}a_ie_i(\sigma)=\sum_{i=1}^{m}a_if_i,\\\notag
XX(\sigma)&=&\sum_{i,j=1}^{m}a_ie_i(a_j)f_j+\sum_{i,j=1}^{m}a_ia_jff_{ij},\\\notag
\nabla_XX &=&\sum_{i,j=1}^{m}a_ie_i(a_j)e_j+\sum_{i,j=1}^{m}a_ia_j\nabla_{e_i}e_j,\\\notag
(\nabla_XX)(\sigma)&=&\sum_{i,j=1}^{m}a_ie_i(a_j)f_j+\sum_{i,j=1}^{m}a_ia_j(\nabla_{e_i}e_j)(\sigma)\\\notag
&=&\sum_{i,j=1}^{m}a_ie_i(a_j)f_j+\sum_{i=1}^{m}f_i^2-\sum_{i,j=1}^{m}a_ia_jf_if_j,\\\label{gd51}
XX(\sigma)&-&(\nabla_XX)(\sigma)-(X \sigma)^2=\sum_{i,j=1}^{m}a_ia_jff_{ij}-\sum_{i=1}^{m}f_i^2.
\end{eqnarray}
Similarly, we have
\begin{eqnarray}\label{gd52}
YY(\sigma)-(\nabla_YY)(\sigma)-(Y \sigma)^2=\sum_{i,j=1}^{m}b_ib_jff_{ij}-\sum_{i=1}^{m}f_i^2.
\end{eqnarray}
Substituting (\ref{gd50}), (\ref{gd51}), (\ref{gd52}) into
(\ref{SC}) we have
\begin{eqnarray*}
K(P)&=\sum_{i,j=1}^{m}(a_ia_j+b_ib_j)ff_{ij}-\sum_{i=1}^{m}f_i^2,
\end{eqnarray*}
which gives  the Lemma.
\end{proof}

Now we are ready to prove the following theorem which provides many counter examples  to the generalized Chen's conjecture.
\begin{theorem}\label{39}
For constants $A, B, C, K,\;$ with $A^2+B^2\ne 0$ we use $\Sigma$ to denote the hyperplane in $\r^5=\{(x_1, \cdots, x_4, z)\}$ defined by $A\sum_{i=1}^{4}x_i+Bz+C=0$. Let $M^4= \r^4\setminus{\Sigma}$ and $C^5=\r^5\setminus{\Sigma}$ be the conformally flat space with the metric $h=(A\sum_{i=1}^{4}x_i+Bz+C)^{-2t}(\sum_{i=1}^m{\rm
d}{x_i}^{2}+{\rm d}{z}^{2})$. Then,\\
\indent $1)$ The isometric immersion $\varphi :
M^4\longrightarrow
(C^5, h=(A\sum_{i=1}^{4}x_i+Bz+C)^{-2t}(\sum_{i=1}^m{\rm
d}{x_i}^{2}+{\rm d}{z}^{2}))$ with
$\varphi(x_1, \cdots, x_4)=(x_1,\cdots, x_4, K)$\;is a proper biharmonic
hypersurface for $t=-1$ or $t=\frac{2A^2}{4A^2+B^2};$\\
\indent $2)$ For $A\ne 0$ and any $t$ with $0<t<1$, the conformally flat space
$(C^5, h=(A\sum_{i=1}^{4}x_i+Bz+C)^{-2t}(\sum_{i=1}^m{\rm
d}{x_i}^{2}+{\rm d}{z}^{2}))$ has negative sectional curvature.
\end{theorem}
\begin{proof}
By Theorem \ref{f31}, we look for the conformal factor $f$ of the form $f(x_1, \cdots, x_4, z)=(A\sum_{i=1}^{4}x_i+Bz+C)^{t}$.
A simple  calculation one checks that for $i, j=1, 2, 3, 4, i\neq j,$

\begin{eqnarray}\label{I}
f_i&=&At(A\sum_{i=1}^{4}x_i+Bz+C)^{t-1},\\\notag f_z&=&Bt(A\sum_{i=1}^{4}x_i+Bz+C)^{t-1},\\\notag
f_{ii}&=&A^2t(t-1)(A\sum_{i=1}^{4}x_i+Bz+C)^{t-2},\\\label{IJ}f_{ij}&=&A^2t(t-1)(A\sum_{i=1}^{4}x_i+Bz+C)^{t-2},\\\notag
f_{iz}&=& A B t(t-1)(A\sum_{i=1}^{4}x_i+Bz+C)^{t-2},\\\notag f_{zz}&=&
B^2t(t-1)(A\sum_{i=1}^{4}x_i+Bz+C)^{t-2},\\\notag
f_{iiz}&=&A^2Bt(t-1)(t-2)(A\sum_{i=1}^{4}x_i+Bz+C)^{t-3}
\end{eqnarray}
Substituting these into the equation (\ref{f32}) we have
$(4A^2+B^2)t^2+(2A^2+B^2)t-2A^2=0$, which has solutions $t=-1$ or
$t=\frac{2A^2}{4A^2+B^2}$. This give the first statement of the theorem.\\

For the second statement, we substitute $f_i, f_{ij}$ in (\ref{I}) and (\ref{IJ}) into the (\ref{38}) to have
\begin{eqnarray*}\notag
K(P)&=&\sum_{i,j=1}^{5}(a_ia_j+b_ib_j)ff_{ij}-\sum_{i=1}^{5}f_i^2\\\notag
&=&A^2\{[(\sum_{i=1}^{5}a_i)^2+(\sum_{i=1}^{5}b_i)^2]t(t-1)-5t^2\}(A\sum_{j=1}^{4}x_j+Bz+C)^{2t-2},
\end{eqnarray*}
which is strictly negative since
$[(\sum_{i=1}^{5}a_i)^2+(\sum_{i=1}^{5}b_i)^2]t(t-1)-5t^2<0$ for
$0<t<1$, and $A^2(A\sum_{j=1}^{4}x_j+Bz+C)^{2t-2}>0$ for
$A\sum_{j=1}^{4}x_j+Bz+C\neq 0$.\\
This completes the proof of the theorem.\\
\end{proof}

To conclude the paper, we would like to point out that by using
Theorem \ref{MT2} to a totally geodesic hypersurface one can very
easily prove the following theorem which was proved in \cite{OT}
using a different method that involves a lengthy computation.
\begin{theorem}\cite{OT}
Let $a_i, i=1, 2, \ldots, m$ and $c$ be constants. Then, the
isometric immersion $\varphi : \mathbb{R}^m\longrightarrow
(\mathbb{R}^{m+1},h=f^{-2}(z)(\sum_{i=1}^m{\rm d}{x_i}^{2}+{\rm
d}{z}^{2})$ with $\varphi(x_1,\ldots, x_m)=(x_1,\ldots,
x_m,\sum_{i=1}^{m}a_ix_i+c)$ is biharmonic  if and only if one of
the following three cases happens
\begin{itemize}
\item[(1)] $f'=0$,
in this case $\varphi$ is minimal (actually, totally geodesic), or
\item[(2)]  $m=4$ and
$f$ is a solution of the equation
\begin{equation}\label{single}
\sum_{i=1}^4a_i^2f^2f'''+(4-\sum_{i=1}^4a_i^2)ff'f''-4(2+\sum_{i=1}^4a_i^2)(f')^3=0,
\end{equation}
or
\item[(3)] $a_i=0,\;i=1,\;\cdots,\;m \;{\rm and}\;
f(z)=\frac{1}{Az+B}$, where  $A$ and $B$ are constants. In this case
each hyperplane is a proper biharmonic hypersurface. This recovers a
result (Theorem 3.1) obtained earlier in \cite{Ou1}.
\end{itemize}
\end{theorem}

\end{document}